\documentclass[10pt,a4paper]{amsart} 
\usepackage{graphicx} 
\usepackage{amssymb} 
\usepackage[dvips=true,bookmarks=true]{hyperref} 
\usepackage[all]{xy}

\theoremstyle{theorem}
\newtheorem{theorem}{Theorem}[section]
\newtheorem{proposition}[theorem]{Proposition}

\newtheorem{conjecture}[theorem]{Conjecture}
\newtheorem*{claim}{Claim}
\newtheorem{lemma}[theorem]{Lemma}
\theoremstyle{remark}

\theoremstyle{definition}
\newtheorem{def.}[theorem]{Definition}

\newtheorem*{acknowledgments}{Acknowledgments}

\def \Z {\mathbb{Z}}

\newcommand\nc{\newcommand}
\nc{\Tors}{\operatorname{Tors}}
\nc{\TorsH}{\Tors H}
\nc\congto{\overset{\cong}{\to}}
\nc\Spin{\operatorname{Spin}}
\nc\Span{\operatorname{Span}}
\nc\SsS{(\Sigma,s_{\Sigma })}
\nc\G{{\mathcal{G}}}
\nc\id{{\operatorname{id}}}
\nc\Ker{\operatorname{Ker}}

\nc\opint{\operatorname{int}}

\title[Borromean surgery equivalence of spin $3$-manifolds]{Borromean surgery equivalence of spin $3$-manifolds with boundary}

\author[Eva Contreras]{Eva Contreras}
\address{Universit\"at Z\"urich, Winterthurerstr. 190
CH-8057 Z\"urich, Switzerland}
\email{eva.tudor@math.uzh.ch}

\author[Kazuo Habiro]{Kazuo Habiro}
\address{Research Institute for Mathematical Sciences, Kyoto
University, Kyoto 606-8502, Japan}
\email{habiro@kurims.kyoto-u.ac.jp}


\date{June 14, 2013 (First version: January 22, 2013)}

\begin{document}

\maketitle

\begin{abstract}
  Matveev introduced Borromean surgery on $3$-manifolds, and proved
  that the equivalence relation on closed, oriented $3$-manifolds
  generated by Borromean surgeries is characterized by the first
  homology group and the torsion linking pairing. Massuyeau
  generalized this result to closed, spin $3$-manifolds, and the second
  author to compact, oriented $3$-manifolds with boundary.
 
  In this paper we give a partial generalization of these results to compact, spin
  $3$-manifolds with boundary.
\end{abstract}

\section{Introduction}

Matveev \cite{Matveev:1986} introduced an equivalence relation on
$3$-manifolds generated by {\em Borromean surgeries}.  This surgery
transformation removes a genus 3 handlebody from a $3$-manifold and glues it
back in a nontrivial, but homologically trivial way.  Thus, Borromean
surgeries preserve the homology groups of $3$-manifolds, and moreover
the torsion linking pairings.  Matveev gave the following
characterization of this equivalence relation.

\begin{theorem}[Matveev \cite{Matveev:1986}]
\label{Matveev}
Two closed, oriented $3$-manifolds $M$ and $M'$ are related by a
sequence of Borromean surgeries if and only if
there is an isomorphism
$f\colon H_1(M;\Z)\to H_1(M';\Z)$ inducing isomorphism on the
torsion linking pairings.
\end{theorem}

Massuyeau \cite{Massuyeau:2003} showed that Borromean surgery induces
a natural correspondence on spin structures, and thus can be regarded
as a surgery move on spin $3$-manifolds.
He generalized Theorem \ref{Matveev} as follows.

\begin{theorem}[Massuyeau \cite{Massuyeau:2003}]
  \label{massuyeau}
  Two closed spin $3$-manifolds $M$ and $M'$ are related by a sequence
  of Borromean surgeries if and only if there is an isomorphism $f\colon
  H_1(M;\Z)\to H_1(M';\Z)$ inducing isomorphism on the torsion linking
  pairings, and the Rochlin invariants of $M$ and $M'$ are congruent
  modulo $8$.
\end{theorem}

In a paper in preparation \cite{Habiro}, the second author generalizes
Matveev's theorem to compact $3$-manifolds with boundary (see Theorem
\ref{habiro} below).

In the present paper, we attempt to generalize the above results to
compact spin $3$-manifolds with boundary.

After defining the necessary ingredients in 
Sections \ref{ysurgery} and \ref{rochlin},
our main result is stated in Theorem \ref{r10}.

\begin{acknowledgments}
  The second author was partially supported by JSPS, Grant-in-Aid for
  Scientific Research (C) 24540077.
  
  The first author would like to thank Anna Beliakova for guidance and
  support, to Christian Blanchet, and to Gw\'{e}na\"{e}l Massuyeau for
  several helpful discussions.  
\end{acknowledgments}

\section{$Y$-surgeries on $3$-manifolds}
\label{ysurgery}

Unless otherwise specified, we will make the following assumptions
in the rest of the paper.  All manifolds are compact and oriented.
Moreover, all $3$-manifolds are connected. All homeomorphisms are
orientation-preserving.  The (co)homology groups with coefficient
group unspecified are assumed to be with coefficients in $\mathbb{Z}$.

\subsection{$Y$-surgeries and $Y$-equivalence}
\label{sec:y-surgeries-y}

Borromean surgery is equivalent to $Y$-surgery used in the theory of
finite type $3$-manifold invariants in the sense of Goussarov and the
second author \cite{Goussarov:1999,Habiro:2000}.

A \emph{$Y$-clasper} in a $3$-manifold $M$ is a connected surface (of genus $0$,
with $4$ boundary components) embedded in $M$,
which is decomposed into one disk, three bands and three annuli as
depicted in Figure ~\ref{fig:Y}.
\begin{figure}[t]
\begin{center}
\includegraphics[height=3.5cm,width=4cm]{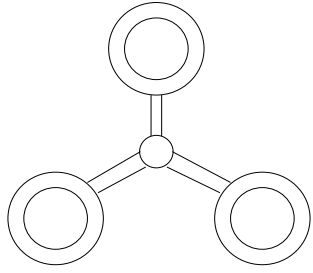} 
\caption{A $Y$-clasper.}
\label{fig:Y}
\end{center}
\end{figure}

We associate to a $Y$-clasper $G$ in $M$ a $6$-component framed link
$L_G$ contained in a regular neighborhood of $G$ in $M$ as depicted in
Figure~\ref{fig:L}.
\begin{figure}[t]
\begin{center}
\includegraphics[height=4cm]{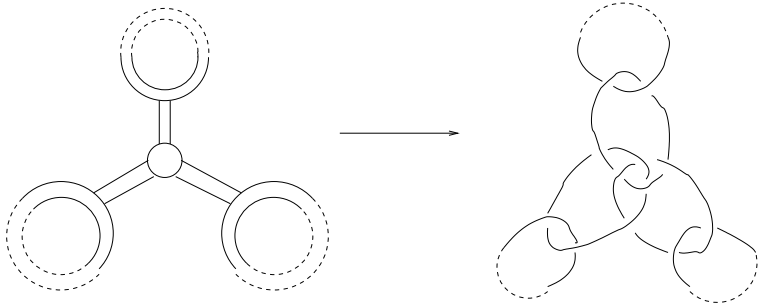} 
\caption{How to replace a $Y$-clasper with a $6$-component framed
  link. Here the framings of the three inner components are zero and
  the framings of the three outer components are determined by the
  annuli in the $Y$-clasper.}
\label{fig:L}
\end{center}
\end{figure}
Surgery along the $Y$-clasper $G$ is defined to be surgery along the
framed link $L_G$.  The result $M_{L_G}$ from $M$ of surgery along
$L_G$ is called the result of {\em surgery along the $Y$-clasper} $G$ and is
denoted by $M_G$.

By {\em $Y$-surgery} we mean surgery along a $Y$-clasper.  Thus, we
say that a $3$-manifold $M'$ is obtained from another $3$-manifold $M$
by a $Y$-surgery if there is a $Y$-clasper $G$ in $M$ such that the result
of surgery, $M_G$, is homeomorphic to $M'$.  It is well-known that
this relation is symmetric, i.e., if $M'$ is obtained from $M$ by a
$Y$-surgery then, conversely, $M$ can be obtained from $M'$ by a
$Y$-surgery.

The {\em $Y$-equivalence} is the equivalence relation on $3$-manifolds
generated by $Y$-surgeries.

\subsection{$\Sigma$-bordered $3$-manifolds}
Throughout the paper, we fix a closed surface $\Sigma$, which may have
arbitrary finite number of components.  In this paper, we consider
$3$-manifolds whose boundaries are parameterized by $\Sigma$.

A {\em $\Sigma$-bordered} $3$-manifold is a pair $(M,\phi)$ of a
compact, connected $3$-manifold $M$ and a homeomorphism 
$\phi\colon\Sigma\congto\partial M$.

Two $\Sigma$-bordered $3$-manifolds $(M,\phi)$ and $(M',\phi')$ are
said to be {\em homeomorphic} if there is a homeomorphism $\Phi\colon
M\congto M'$ such that $(\Phi|_{\partial M})\circ\phi=\phi'$.

\subsection{$Y$-equivalence for $\Sigma$-bordered $3$-manifolds}

The notions of $Y$-surgery and $Y$-equivalence extend to
$\Sigma$-bordered $3$-manifolds in a natural way.

For a $\Sigma$-bordered $3$-manifold
$(M,\phi\colon\Sigma\congto\partial M)$ and a $Y$-clasper $G$ in $M$,
the result of surgery $M_G$ has an obvious boundary parameterization
$\phi_G\colon \Sigma\congto \partial M_G$ induced by $\phi$.  Thus
surgery along a $Y$-clasper $G$ in a $\Sigma$-bordered
$3$-manifold $(M,\phi)$ yields a $\Sigma$-bordered $3$-manifold
$(M,\phi)_G:=(M_G,\phi_G)$.  Two $\Sigma$-bordered $3$-manifolds
$(M,\phi)$ and $(M',\phi')$ are said to be related by a $Y$-surgery if there
is a $Y$-clasper $G$ in $M$ such that 
$(M,\phi)_G$ is homeomorphic to $(M',\phi')$.
The $Y$-equivalence on $\Sigma$-bordered $3$-manifolds is
generated by $Y$-surgeries.

The following well known characterization of the $Y$-equivalence is useful.

\begin{lemma}
  \label{r1}
  Two $\Sigma$-bordered $3$-manifolds $(M,\phi)$ and $(M',\phi')$ are
  $Y$-equivalent if and only if there are finitely many, mutually
  disjoint $Y$-claspers $G_1,\ldots,G_n$ ($n\ge0$) in $M$ such that
  the result of surgery, $(M,\phi)_{G_1,\ldots,G_n}$ is homeomorphic
  to $(M',\phi')$.
\end{lemma}

\subsection{Homology isomorphisms between compact $3$-manifolds}
\label{sec:enhanc-homol-isom}
Let $(M,\phi)$ and $(M',\phi')$ be $\Sigma$-bordered $3$-manifolds.
Set 
\begin{gather*}
  \delta:=\phi'\circ\phi^{-1}\colon \partial M\congto \partial M'.
\end{gather*}

A \emph{homology isomorphism}\footnote{In \cite{Habiro}, this is
  called ``full enhanced homology isomorphism''.  In this paper, we
  call it ``homology isomorphism'' for simplicity.} from $(M,\phi)$ to
  $(M',\phi')$ (or a 
homology isomorphism from $M$ to $M'$ along $\delta$) is an
isomorphism $f=(f_i, \overline{f_i})_{i=0,1,2,3}$ of the homology
exact sequences of pairs $(M,\partial M)$ and $(M',\partial M')$
\begin{equation*}
\xymatrix@R=6mm@C=3mm{
\cdots \ar[r]& H_i(\partial M)\ar[d]^{\delta_*} \ar[r] & H_i(M) \ar[r] \ar[d]^{f_i} & H_i(M,\partial M) \ar[r] \ar[d]^{\overline{f}_i}& H_{i-1}(\partial M) \ar[r] \ar[d]^{\delta_*}& \cdots\\
\cdots \ar[r]& H_i(\partial M') \ar[r] & H_i(M') \ar[r] & H_i(M',\partial M') \ar[r] & H_{i-1}(\partial M') \ar[r]& \cdots \\
}
\end{equation*}
satisfying the following properties:
\begin{enumerate}
 \item [(i)]$f_0([pt])=[pt]$;
\item [(ii)]$f_i$ and $\overline{f_i}$ are compatible with the
intersection forms, i.e., for $i=0,1,2,3$, the square commutes:
\begin{equation*}
\xymatrix@R=6mm@C=10mm{
H_i(M)\times H_{3-i}(M,\partial M)\ar[d]^{f_i\times \overline{f}_{3-i}} \ar[r]^-{\langle,\rangle_M} &  \mathbb{Z} \ar@{=}[d]\\
H_i(M')\times H_{3-i}(M',\partial M')\ar[r]^-{\langle,\rangle_{M'}} &  \mathbb{Z}
}
\end{equation*}

Here $\langle,\rangle_M$ and $\langle,\rangle_{M'}$ denote the intersection forms.

\item [(iii)]$f_1$ and $\overline{f_1}$ are compatible with the
torsion linking pairings, i.e., the square commutes:
\begin{equation*}
\xymatrix@R=6mm@C=8mm{
\TorsH_1(M)\times \TorsH_{1}(M,\partial M) \ar[r]^-{\tau_M}
\ar[d]^{\Tors f_1\times \Tors\overline{f}_1} &  \mathbb{Q}\diagup \mathbb{Z}\ar@{=}[d]\\
\TorsH_1(M')\times \TorsH_1(M',\partial M')\ar[r]^-{\tau_{M'}} & \mathbb{Q}\diagup \mathbb{Z}.
}
\end{equation*}
Here $\Tors$ denotes torsion part, and $\tau_M$ denotes the
torsion linking pairing of $M$.
\end{enumerate}

The classification of compact $3$-manifolds up to $Y$-equivalence is
given by the following result.

\begin{theorem}[\cite{Habiro}]
\label{habiro}
Let $\Sigma$ be a closed surface, and let $(M,\phi)$ and $(M',\phi')$
be two $\Sigma$-bordered $3$-manifolds.  Then the
following conditions are equivalent.
\begin{enumerate}
\item $(M,\phi)$ and $(M',\phi')$ are $Y$-equivalent.
\item There is a homology isomorphism from $(M,\phi)$ to $(M',\phi')$.
\end{enumerate}
\end{theorem}

For closed $3$-manifolds, Theorem \ref{habiro} is equivalent to
Matveev's theorem (Theorem \ref{Matveev}).

\section{$Y$-surgery on spin $3$-manifolds}
\label{rochlin}

\subsection{Spin structures}

For an oriented manifold $M$ with vanishing second Stiefel-Whitney
class, let $\Spin(M)$ denote the set of spin structures on $M$.  

It is well known that $\Spin(M)$ is affine over $H^1(M; \Z_2)$, i.e,
acted by $H^1(M; \Z_2)$ freely and transitively
\begin{gather*}
  \Spin(M)\times H^1(M; \Z_2) \to \Spin (M),\quad (s,c)\mapsto s+c.
\end{gather*}

An embedding $f\colon M'\hookrightarrow M$ of a manifold $M'$ into $M$
induces a map
\begin{gather*}
  i^*\colon \Spin(M)\to \Spin(M').
\end{gather*}
If $i$ is an inclusion map, $i^*(s)$, $s\in\Spin(M)$, is denoted also
by $s|_{M'}$.

\subsection{$Y$-surgery and spin structures}

Let $G$ be a $Y$-clasper in a $3$-manifold $M$.  Let $N(G)$ be a
regular neighborhood of $G$ in $M$.  Note that the result of surgery,
$M_G$, can be identified with the manifold
\begin{gather*}
  (M\setminus \operatorname{int}N(G))\cup_{\partial N(G)}N(G)_G
\end{gather*}
obtained by gluing $M\setminus \operatorname{int}N(G)$ with
$N(G)_G$ along $\partial N(G)$.

As is proved by Massuyeau \cite{Massuyeau:2003}, for a spin
structure $s\in\Spin(M)$, there is a unique spin structure $s_G$ on
$M_G$ such that 
\begin{gather*}
  s_G|_{M\setminus \operatorname{int}N(G)}=s|_{M\setminus \operatorname{int}N(G)}.
\end{gather*}
This gives a bijection
\begin{gather*}
  \Spin(M)\congto \Spin(M_G),\quad s\longmapsto s_G.
\end{gather*}
The spin $3$-manifold $(M_G,s_G)$ is called the result of surgery on the spin
$3$-manifold $(M,s)$ along $G$.

As in Section \ref{sec:y-surgeries-y}, the {\em $Y$-equivalence} on spin
$3$-manifolds is the equivalence relation generated by $Y$-surgery.

\subsection{Twisting a spin structure along an orientable surface}
\label{sec:twist-spin-struct}

Let $(M,s)$ be a spin $3$-manifold possibly with boundary, and let $T$
be an orientable surface properly embedded in $M$.  Then we can {\em
twist} the spin structure $s$ along $T$.  More precisely, we can
define a new spin structure
\begin{gather*}
  s*T= s+[T]^!\in\Spin(M),
\end{gather*}
where $[T]^!\in H^1(M;\Z_2)$ is the Poincar\'{e} dual of $[T]\in
H_2(M,\partial M;\Z_2)$.  (One can consider similar operation when $T$
is non-orientable, but  we do not need it in this paper.)

Note that twisting along a closed surface preserves the restriction of the
spin structure to the boundary.

\begin{proposition}
  \label{r6}
If $T$ is a closed, orientable surface in a spin $3$-manifold $(M,s)$,
then $(M,s*T)$ is $Y$-equivalent to $(M,s)$.
\end{proposition}

\begin{proof}
  We may assume that $T$ is connected, since the general case follows
  from this special case.

  Take a bicollar neighborhood $T\times[-1,2]\subset M$.
  Set $T_2=T\times\{2\}\subset M$.
  Let $c$ be a simple closed curve in $T$ bounding a disk in $T$.
  Let $A$ denote a bicollar neighborhood of $c$ in $T$.
  Let $D$ and $T'$ be the two components of  $T\setminus \opint A$,
  where $D$ is a disk.
  Set
  \begin{gather*}
    V_0= A\times[-1,1],\quad
    V_1= (A\cup D)\times[-1,1],\quad
    V_2= (A\cup T')\times[-1,1],\\
    M_i=M\setminus \opint V_i,\quad i=0,1,2.
  \end{gather*}
  Note that $M_1,M_2\subset M_0$.
  For $i=0,1,2$, set $s_i=s|_{M_i}\in\Spin(M_i)$.

  Let $K=(c,+1)$ denote the framed knot in $M$ whose underlying knot
  is $c$ and the framing is $+1$.  Let $M_K$ denote the result of
  surgery along $K$, which may be regarded as the manifold
  $M_0\cup_\partial (V_0)_K$ obtained from $M_0$ and the result of surgery
  $(V_0)_K$ by gluing along their boundaries in the natural way.
  We may regard $M_0$, $M_1$ and $M_2$ as submanifolds of $M_K$.

  Note that $V_1$ and $(V_1)_K$ are $3$-balls.  Hence there is a
  unique spin structure $s_K\in\Spin(M_K)$ such that
  $(s_K)|_{M_1}=s_1$.  We have the spin homeomorphism
  $(M,s)\cong(M_K,s_K)$.

  We have
  \begin{gather*}
    s_K|_{M_0}=s_0*D=s_0*T_2.
  \end{gather*}
  Hence we have
  \begin{gather*}
    s_K|_{M_2} =s_2*T_2.
  \end{gather*}

  It suffices to prove that $(M_K,s_K)$ is $Y$-equivalent to
  $(M,s*T_2)=(M,s*T)$.  Since the framed knot $K$ is
  null-homologous in $V_2$ and $+1$-framed, $(V_2)_K$ is
  $Y$-equivalent to $V_2$ in a way respecting the boundary
  \cite{Matveev:1986}.  This $Y$-equivalence extends to
  $Y$-equivalence of $M_K$ and $M$.  This $Y$-equivalence implies the
  desired $Y$-equivalence of $(M_K,s_K)$ and $(M,s*T_2)$ since we
  have
  \begin{gather*}
    s_K|_{M_2}=s_2*T_2=(s*T_2)|_{M_2},
  \end{gather*}
  and since the maps
  \begin{gather*}
    \Spin(M_K)\to\Spin(M_2),\quad \Spin(M)\to\Spin(M_2)
  \end{gather*}
  induced by inclusions are injective.
\end{proof}

\subsection{$(\Sigma,s_\Sigma)$-bordered spin $3$-manifolds}

We fix a spin structure $s_\Sigma\in\Spin(\Sigma)$.  In the following
we consider $Y$-equivalence of spin $3$-manifolds with boundary
parameterized by the spin surface $\SsS$.

A {\em $\SsS$-bordered spin $3$-manifold} is a triple $(M,\phi ,s)$
consisting of a $\Sigma$-bordered $3$-manifold $(M,\phi)$ and a spin
structure $s\in\Spin(M)$ such that $\phi^*(s)=s_\Sigma $.

Clearly, surgery along a $Y$-clasper in $M$ preserves the spin
structure on the boundary of $M$.  Hence a $Y$-surgery on a
$\SsS$-bordered spin $3$-manifold yields another $\SsS$-bordered spin
$3$-manifold.

\subsection{Gluing of $\SsS$-bordered spin $3$-manifolds}

Let $(M,\phi,s)$ and $(M',\phi',s')$ be two $\SsS$-bordered spin
$3$-manifolds.   Let $M''=(-M)\cup_{\phi,\phi'} M'$ be the closed
$3$-manifold obtained from $-M$ (the orientation reversal of $M$) and
$M'$ by gluing their boundaries along $\phi'\circ \phi^{-1}$.

By a {\em gluing} of $s$ and $s'$, we mean a spin structure
$s''\in\Spin(M'')$ satisfying
\begin{gather*}
  s''|_{-M}=s,\quad s''|_{M'}=s'.
\end{gather*}
If $\Sigma$ is empty or connected, then $s''$ is uniquely determined
by $s$ and $s'$.  Otherwise, $s''$ is not unique.

The spin manifold $(M'',s'')$ is called a {\em gluing} of $(M,\phi,s)$
and $(M',\phi',s')$.

\begin{proposition}
  \label{r5}
  All the gluings of two $\SsS$-bordered spin $3$-manifolds $(M,\phi,s)$
  and $(M',\phi',s')$ are mutually $Y$-equivalent.
\end{proposition}

\begin{proof}
  If $\Sigma$ has at most one boundary component, then there is
  nothing to prove since there is only one gluing of $(M,\phi,s)$
  and $(M',\phi',s')$.
  
  Suppose $\Sigma$ has components $\Sigma_1,\dots,\Sigma_n$ with
  $n\ge2$.  For $i=2,\dots,n$, choose a framed knot $K_i$ in
  $M''=(-M)\cup_{\phi,\phi'} M'$ which transversely intersects each of
  $\Sigma_1$ and $\Sigma_i$ by exactly one point and is disjoint from
  the other components of $\Sigma$.  There are $2^{n-1}$
  gluings $s''_{\epsilon_2,\dots,\epsilon_n}\in\Spin(M'')$ of $s$ and
  $s'$ for $\epsilon_2,\dots,\epsilon_n\in\{0,1\}$, where for
  $i=2,\dots,n$ the framed knot $K_i$ is even framed with
  respect to $s''_{\epsilon_2,\dots,\epsilon_n}$ if $\epsilon_i=0$,
  and odd framed otherwise.  Moreover, we have
  \begin{gather*}
    s''_{\epsilon_2,\dots,\epsilon_n}
    =s''_{0,\dots,0}*\left(\bigcup_{2\le i\le n,\;\epsilon_i=1}\Sigma_i\right).
  \end{gather*}
  Hence, by Proposition \ref{r6}, $(M'',s''_1)$ and $(M'',s''_2)$ are
  $Y$-equivalent.
\end{proof}

\subsection{Rochlin invariant mod $8$ of pairs of $\SsS$-bordered spin $3$-manifolds}
\label{sec:rochl-invar-pairs}

Let $(M,\phi,s)$ and $(M',\phi',s')$ be two $\SsS$-bordered spin
$3$-manifolds.  Set
\begin{gather}
  \label{e1}
  R_8((M,\phi,s),(M',\phi',s')):= (R(M'',s'')\mod {8})\in\Z_8,
\end{gather}
where $M''=(-M)\cup_{\phi,\phi'} M'$ as before and $s''\in\Spin(M'')$
is any gluing of $s$ and~$s'$.  Proposition \ref{r5} and Theorem
\ref{massuyeau} imply that \eqref{e1} is well defined.

\begin{lemma}
  \label{r3}
  The invariant $R_8((M,\phi,s),(M',\phi',s'))$ depends only 
  on the $Y$-equivalence classes of $(M,\phi,s)$ and $(M',\phi',s')$.
\end{lemma}

\begin{proof}
  Suppose that $(M_1,\phi_1,s_1)$ is $Y$-equivalent to
  $(M_2,\phi_2,s_2)$ and that $(M'_1,\phi'_1,s'_1)$ is $Y$-equivalent to
  $(M'_2,\phi'_2,s'_2)$.  Consider gluings $(M''_i,s''_i)$ of
  $(M_i,\phi_i,s_i)$ and $(M'_i,\phi'_i,s'_i)$ for $i=1,2$.
  Then $(M''_1,s''_1)$ and $(M''_2,s''_2)$ are $Y$-equivalent.
  Hence we have
  \begin{gather*}
    \begin{split}
      &R_8((M_1,\phi_1,s_1),(M'_1,\phi'_1,s'_1))
    =(R(M''_1,s''_1)\mod8)\\
    &\quad \quad =(R(M''_2,s''_2)\mod8)
    =R_8((M_2,\phi_2,s_2),(M'_2,\phi'_2,s'_2)).
    \end{split}
  \end{gather*}
\end{proof}

\subsection{Main results}

Now we state the main result of the present paper, which gives a
characterization of $Y$-equivalence of $\SsS$-bordered spin
$3$-manifolds in terms of homology isomorphism and the Rochlin
invariant mod $8$.

\begin{conjecture}
\label{main}
Let $(M,\phi,s)$ and
$(M',\phi',s')$ be two $\SsS$-bordered spin $3$-manifolds.
Then the following conditions are equivalent.
\begin{enumerate}
\item $(M,\phi,s)$ and $(M',\phi',s')$ are $Y$-equivalent.
\item There is a homology isomorphism from $(M,\phi)$ to $(M',\phi')$,
and we have
\begin{gather*}
  R_8((M,\phi,s),(M',\phi',s'))=0\pmod {8}.
\end{gather*}
\end{enumerate}
\end{conjecture}

It follows from Theorem \ref{habiro} that Conjecture \ref{main} is equivalent to the
following.

\begin{conjecture}
  \label{r2}
  Let $(M,\phi,s)$ and
  $(M',\phi',s')$ be two $\SsS$-bordered spin $3$-manifolds.
  Then the following conditions are equivalent.
  \begin{enumerate}
  \item $(M,\phi,s)$ and $(M',\phi',s')$ are $Y$-equivalent.
  \item $(M,\phi)$ and $(M',\phi')$ are $Y$-equivalent, and we have
    \begin{gather*}
      \label{e5}
      R_8((M,\phi,s),(M',\phi',s'))=0\pmod {8}.
    \end{gather*}
  \end{enumerate}
\end{conjecture}

The following theorem says that Conjecture \ref{r2} holds when
$H_1(M;\Z)$ has no $2$-torsion.  The proof of this result does not use
definitions and results given in \cite{Habiro}, which is not 
available when we are writing the present paper.

\begin{theorem}
  \label{r10}
  In the setting of Conjecture \ref{r2}, (1) implies (2).
  Moreover, if $H_1(M;\Z)$ has no $2$-torsion, then 
  \begin{enumerate}
     \item[(2')] $(M,\phi)$ and $(M',\phi')$ are $Y$-equivalent.
  \end{enumerate}
  implies (1).
\end{theorem}

\section{Proof of Theorem \ref{r10}}
\label{sec:proof-theor-refm}

\subsection{Proof of $(1)\Rightarrow(2)$}

Suppose that (1) of Theorem \ref{r2} holds.  Then, clearly, $(M,\phi)$
and $(M',\phi')$ are $Y$-equivalent.  We have to prove
$R(M'',s'')\equiv 0\pmod8$, where $(M'',s'')$ is a gluing of
$(M,\phi,s)$ and $(M',\phi',s')$.

Since $(M,\phi,s)$ and $(M',\phi',s')$ are $Y$-equivalent, Lemma
\ref{r3} implies that $(M'',s'')$ is $Y$-equivalent to a gluing
$(M''_0,s''_0)$ of $(M,s)$ and itself.

Consider the $4$-manifold $C$ which is the quotient of the cylinder
$M\times[0,1]$ by the equivalence relation $(x,t)\sim(x,t')$ for
$x\in\partial M$ and $t,t'\in[0,1]$.  Then we may naturally identify
$M''_0$ with $\partial C$.  The $4$-manifold $C$ has a spin structure 
$s_C$ induced by the spin structure $s\times s_{[0,1]}\in\Spin(M\times
[0,1])$, where $s_{[0,1]}$ is the unique spin structure of $[0,1]$.
We have 
\begin{gather*}
  R(C,s_C)\equiv R(M\times[0,1],s\times s_{[0,1]})\equiv\sigma(M\times[0,1])=0\pmod{16}.
\end{gather*}

Since both $s''_0$ and $s_C$ are gluings of $(M,s)$ and itself, Proposition
\ref{r5} implies that $(M''_0,s''_0)$ and $(C,s_C)$ are
$Y$-equivalent.
Hence, by Theorem \ref{massuyeau}, we have
\begin{gather*}
  R(M'',s'')\equiv R(M''_0,s''_0)\equiv R(C,s_C)\equiv 0 \pmod8.
\end{gather*}

\subsection{Proof of $(2')\Rightarrow(1)$ when $H_1(M;\Z)$ has no $2$-torsion}

We assume that $H_1(M;\Z)$ has no $2$-torsion.

We divide the proof into three cases:
\begin{itemize}
\item $M$ is a $\Z_2$-homology handlebody, i.e.,
  $\partial M$ is connected and $H_1(M,\partial
  M;\Z_2)=0$.
\item $M$ has non-empty boundary.
\item $M$ is closed.
\end{itemize}

\subsubsection{Case where $M$ is a $\Z_2$-homology handlebody}
\label{sec:case-where-m}
Since $\Spin(M)\overset{\phi^*}{\to}\Spin(\Sigma)$ and
$\Spin(M')\overset{(\phi')^*}{\to}\Spin(\Sigma)$ are injective,
$Y$-equivalence of $(M,\phi)$ and $(M',\phi')$ implies $Y$-equivalence
of $(M,\phi,s)$ and $(M',\phi',s')$.

\subsubsection{Case where $\partial M$ is non-empty}
\label{sec:case-where-sigma}
We will use the following result.
\begin{lemma}
 Let $M$ be a $3$-manifold with boundary such that $H_1(M;\Z)$ has no $2$-torsion.
Then $M$ can be obtained from a $\Z_2$-homology handlebody $V$ by attaching $2$-handles
$h_1,\ldots, h_n$ (with $n\geq 0$) along simple closed curves
 $c_1,\ldots,c_n$ in $\partial V$ in such a way that each $c_i$ is
 null-homologous (over $\Z$) in $V$.
\end{lemma}

\begin{proof}
 $M$ can be obtained from a solid torus $V'$ of genus $g$ by attaching some $2$-handles
along simple closed curves $c'_1,\ldots, c'_k$ in $\partial V'$. After finitely many
handle-slides, we can assume the following.
\begin{itemize}
\item There is a basis $x_1, \ldots,x_g$ of $H_1(V';\Z)$ such that we have 
  \begin{gather*}
    [c_i]=\sum_{j=1}^g a_{i,j}x_j
  \end{gather*}
  for $i=1,\ldots,k$, where the matrix $(a_{i,j})$ is diagonal (but not
  necessarily square), in the sense that $a_{i,j}=\delta_{i,j}d_i$.
\end{itemize}
Clearly, $H_1(M;\Z)$ is isomorphic to
$\bigoplus_{i=1}^k\Z_{d_i}$. By the assumption that $H_1(M;\Z)$ has no $2$-torsion,
each $d_i$ is either odd or $0$. 

We may assume that, for some $n$, we have $d_1=\cdots=d_n=0$ and $d_{n+1}, \ldots,d_k$ are odd.
The union $V:=V'\cup h'_{n+1}\cup\cdots\cup h'_k$ is a $\Z_2$-homology 
handlebody. Setting $c_i=c'_i, h_i=h'_i$ for $i=1,\ldots,n$, we have the result.

\end{proof}

Let $M$ be obtained as above from a $\Z_2$-homology handlebody $V$ 
by attaching $2$-handles $h_1,\dots,h_n$ along disjoint simple closed curves
$c_1,\dots,c_n\subset\partial V$, $n=\operatorname{rank}H_2(M;\Z)\ge0$,
such that $c_i$ is null-homologous in $M$ and such that $\partial
M\setminus(c_1\cup\dots\cup c_n)$ is connected.  

The proof is by induction on $n$.  The case $n=0$ is proved in Section
\ref{sec:case-where-m}.  Suppose $n>0$.

Let $N=h_n=D^2\times[0,1]\subset M$ be one of the $2$-handles.
Set 
\begin{gather*}
A=\partial D^2\times[0,1]\subset \partial N,\\
B=D^2\times\{0,1\}\subset \partial N,\\
M_0 := \overline{M\setminus N}=V\cup h_1\cup\dots\cup h_{n-1}\subset M.
\end{gather*}
Thus, $M=M_0\cup_A N$ is obtained from a $3$-manifold $M_0$ by attaching $N$
along an annulus $A\subset\partial M_0$.  

Since $(M,\phi)$ and $(M',\phi')$ are $Y$-equivalent, it follows from
Lemma \ref{r1} that there exists a
disjoint family $\G$ of $Y$-claspers in $M$ and a homeomorphism
\begin{gather*}
  \Psi\colon (M_\G,\phi_\G)\congto(M',\phi').
\end{gather*}
By isotoping $\G$ if necessary, we may assume that $\G$ is contained
in the interior of~$M_0$.

Set $\Sigma_0:=(\Sigma\setminus \opint(\phi^{-1}(B)))\cup A$.  Then we
have a $\Sigma_0$-bordered $3$-manifold $(M_0,\phi_0)$ where
$\phi_0\colon\Sigma_0\congto\partial M_0$ is obtained by gluing
$\phi|_{\Sigma\setminus \opint(\phi^{-1}(B))}$ and $\id_A$.

Set 
\begin{gather*}
  M'_0:=\Psi((M_0)_\G)=\overline{M'\setminus \Psi(N)}\subset M'.
\end{gather*}
We have a $\Sigma_0$-bordered $3$-manifold $(M_0',\phi_0')$, where
$\phi_0'\colon\Sigma_0\congto\partial M'_0$ is obtained by gluing
$\phi'|_{\Sigma\setminus \opint(\phi^{-1}(B))}$ and $\Psi|_A\colon
A\congto \Psi(A)$.

We have  a homeomorphism of $\Sigma_0$-bordered $3$-manifolds
\begin{gather*}
  \Psi_0:=\Psi|_{M_0}\colon ((M_0)_\G,(\phi_0)_\G)\congto(M'_0,\phi'_0).
\end{gather*}

Set $s_{\Sigma_0}=(\phi_0)^*(s|_{M_0})\in\Spin(\Sigma_0)$ and
$s'_{\Sigma_0}=(\phi'_0)^*(s'|_{M_0'})\in\Spin(\Sigma_0)$.
Note that $s_{\Sigma_0}|_{\Sigma\setminus \opint(\phi^{-1}(B))} =
s'_{\Sigma_0}|_{\Sigma\setminus \opint(\phi^{-1}(B))}$.
Hence we have either
\begin{gather}
  \label{e3}
  s_{\Sigma_0}=s'_{\Sigma_0}
\end{gather}
or
\begin{gather}
  \label{e4}
  s'_{\Sigma_0}=s_{\Sigma_0}+[a]^!\quad \text{and}\quad
  s_{\Sigma_0}\neq s'_{\Sigma_0},
\end{gather}
where $a=c_n=\partial D^2\times \{1/2\}\subset A$ is the core of the
annulus $A$, and $[a]^!\in H^1(\Sigma_0;\Z_2)$ is the Poincar\'{e}
dual to $[a]\in H_1(\Sigma_0;\Z_2)$.

\begin{claim}
  We may assume \eqref{e3}.
\end{claim}

\begin{proof}
  If $a$ is separating in $\Sigma_0$, then we have \eqref{e3}.

  Suppose that $a$ is non-separating in $\Sigma_0$, and that we have
  \eqref{e4}.  Since $a$ is null-homologous in $\partial V\subset
  M_0$, it is so also in $M_0'$.  Therefore, there is a connected,
  oriented surface $T_0'$ properly embedded in $M_0'$ such that
  $\partial T_0' = a$.  Set $D'=\Psi(D^2\times \{1/2\})$, and
  $T'=T_0'\cup D'$, which is a connected, oriented, closed surface in
  $M'$.

  Set $\hat{s}':=s'*T'\in\Spin(M')$ and 
  $\hat{s}'_{\Sigma_0}
  =(\phi'_0)^*(\hat{s}'|_{M_0'})\in\Spin(\Sigma_0)$.
   By Proposition \ref{r6}, it
  follows that $(M',s')$ and $(M',\hat{s}')$ are $Y$-equivalent.
  Thus, we may replace the spin manifold $(M',s')$ with $(M',\hat{s}')$.
  We have
  \begin{gather*}
    \hat{s}'_{\Sigma_0}
    =(\phi'_0)^*((s'*T')|_{M_0'})
    =(\phi'_0)^*(s')+[a]^!
    =s'_{\Sigma_0}+[a]^!
    =s_{\Sigma_0}.
  \end{gather*}
  Hence, we have only to consider the case where \eqref{e3} holds.
\end{proof}

We assume \eqref{e3}.  Set $s_0=s|_{M_0}\in\Spin(M_0)$ and
$s'_0=s'|_{M'_0}\in\Spin(M'_0)$.  Then $(M_0,\phi_0,s_0)$ and
$(M_0',\phi_0',s_0')$ are $(\Sigma_0,s_{\Sigma_0})$-bordered spin
$3$-manifolds.

We can use the induction hypothesis to deduce that $(M_0,\phi_0,s_0)$ and
$(M_0',\phi_0',s_0')$ are $Y$-equivalent, and hence so are $(M,\phi,s)$ and $(M',\phi',s')$.

\subsubsection{Case where $M$ is closed}
\label{sec:case-where-m-1}

This case is a special case of Theorem \ref{massuyeau}.

Alternatively, this case easily follows from the previous case by
considering the punctures $M\setminus \opint B^3$ and $M'\setminus \opint
B^3$.


\begin{thebibliography}{14}
\providecommand{\natexlab}[1]{#1}
\providecommand{\url}[1]{\texttt{#1}}
\expandafter\ifx\csname urlstyle\endcsname\relax
  \providecommand{\doi}[1]{doi: #1}\else
  \providecommand{\doi}{doi: \begingroup \urlstyle{rm}\Url}\fi



\bibitem{Goussarov:1999}
M.~Goussarov.
\newblock \emph{Finite type invariants and $n$-equivalence of
    $3$-manifolds}.
\newblock {Compt. Rend. Acad. Sci. Paris}  S\'erie I {\bf 329} (1999),
517--522.

\bibitem{Habiro:2000}
K.~Habiro.
\newblock \emph{{Claspers and finite type invariants of links}}.
\newblock {Geom. Topol.} {\bf 4} (2000), 1--83.

\bibitem{Habiro}
K.~Habiro.
\newblock{in preparation}.

\bibitem{Massuyeau:2003}
G.~Massuyeau.
\newblock \emph{{Spin Borromean surgeries}}.
\newblock {Trans. Amer. Math. Soc.} {\bf 10} (2003),  3991--4016.


\bibitem{Matveev:1986}
S.~V. Matveev.
\newblock \emph{{Generalized surgery of three-dimensional manifolds and representations of homology spheres}}.
\newblock  \emph{Mat. Zametki} {\bf 42} (1987), 268--278, 345


\end{thebibliography}
\end{document}